\documentclass[12pt]{amsart}
\usepackage{geometry}                
\usepackage[parfill]{parskip}    
\usepackage{graphicx}
\usepackage{amssymb, amsthm, amsfonts, amssymb}

\newtheorem{thm}{Theorem}
\newtheorem{lemma}{Lemma}[section]

\newtheorem{defn}{Definition}[section]
\newtheorem{prop}{Proposition}[section]

\theoremstyle{definition}

\newcommand{\ud}{\,\mathrm{d}}
\newcommand{\La}{\Lambda}

\newcommand{\om}{\omega}
\newcommand{\Om}{\Omega}

\newcommand{\eps}{\epsilon}
\newcommand{\Del}{\mathfrak{M}}

\author{Thomas Beck \and Philippe Sosoe \and Percy Wong}
\address{Department of Mathematics, Princeton University\\ Fine Hall, Washington Road, Princeton NJ.}
\email{tdbeck@math.princeton.edu}
\email{psosoe@math.princeton.edu}
\email{pakwong@math.princeton.edu}

\title[Duchon-Robert solutions]{Duchon-Robert solutions for the Rayleigh-Taylor and Muskat problems}
\date{}
\begin{document}
\maketitle
\begin{abstract}
We construct analytic solutions to the Euler equations with an interface between two fluids, extending work of Duchon and Robert. We also show that the estimates of Duchon and Robert yield global analytic solutions to the Muskat problem with small initial data.
\end{abstract}
\section{Introduction}
We consider the interface problem for two perfect incompressible fluids in $\mathbb{R}^2$ under the influence of gravity. Each of the two fluids occupies one of the two connected regions $\Omega_\pm(t)$ in the complement of a moving interface $\Sigma(t)$ parametrized by a curve $y(x_1,t)$: 
\[\Sigma(t) = \{(x_1,x_2): x_2 = y(x_1,t)\}.\]
$\Omega_+(t)$ is the region above the curve $\Sigma(t)$
\[\Omega_+(t)=\{(x_1,x_2):x_2>y(x_1,t)\},\]
while $\Omega_-(t)=\{(x_1,x_2):x_2<y(x_1,t)\}$ is the region below.
The equations of motion are given by the two-dimensional Euler system:
\begin{equation}\label{VS}
 \Bigg\{\begin{array}{rl}
    \rho_{\pm}v_t + \nabla\cdot (\rho_\pm v\otimes v) + \nabla p &= \rho_\pm g \mathbf{e}_y \quad  \text{in }
\Omega_\pm \\
    \textrm{div }v &= 0 \quad  \text{in } \Omega_\pm.
   \end{array}
\end{equation}
Here, $v=(v_1,v_2)$ is the velocity of the fluid, $\rho_+$ and  $\rho_-$ are the constant densities of the two fluids in the regions $\Omega_+$ and $\Omega_-$, respectively. The constant $g<0$ is the gravitational acceleration, and $\mathbf{e}_y=(0,1)$. It is useful to introduce the Atwood number $a$, defined as 
\[a = \frac{\rho_+-\rho_-}{\rho_++\rho_-}.\]  

We are interested in \emph{vortex sheet} solutions to the system (\ref{VS}). These are solutions such that, for all times $t>0$, the vorticity 
\[\Omega = \operatorname{curl}v\]
is a measure supported on the interface curve $\Sigma(t)$. In \cite{DR}, the authors construct global solutions for small initial data to the vortex sheet problem in a homogeneous fluid. Homogeneity corresponds to the case $\rho_+ = \rho_-$, $a=0$. In the present work, we show that their method extends to the case $a<0$ when the fluid with higher mass density lies above the lower-density fluid. When $a>0$ (``light fluid on top''), we obtain local-in-time solutions. Even when they are only known to exist for short time, these special solutions are of interest because they provide examples of vortex sheets for which, at the initial time, the interface has limited regularity ($x\mapsto y(x,0)$ is not in $C^{1+\alpha}$ for any $\alpha>0$), but is analytic for any positive time. Note that we are able to construct global solutions in the physically ``unstable'' case, while in the physically ``stable'' case, we only obtain local solutions. The Rayleigh-Taylor problem is known to be ill-posed for both $a>0$ and $a<0$; the physical relevance of the solutions we find is unclear.

Since the paper of Duchon and Robert, several works on the ill-posedness of the Kelvin-Helmholtz ($a=0$) and Rayleigh-Taylor ($a\neq0$) problems have appeared. We mention a selection of them here. In \cite{KL}, Kamotski and Lebeau prove that for the Rayleigh-Taylor problem, if the vorticity density and interface have $C^{1+\alpha}$ regularity locally in space and time for some $\alpha>0$, then in fact they must be $C^\infty$ smooth. In \cite{W}, Wu considers the Kelvin-Helmholtz problem and shows that if the vorticity density is both bounded and bounded away from zero, and if the interface satisfies the chord-arc condition, then the vorticity density and interface must be analytic. For an excellent survey on problems related to interfaces in two-dimensional fluids, we refer the reader to the survey paper, \cite{BL}, by Bardos and Lannes.

In the final section of this paper, we show that a straightforward application of the method of Duchon and Robert yields global analytic solutions to another interface dynamics problem, the Muskat equation. This models the evolution of the interface between two fluids of different densities in a porous medium. See \cite{CCGS} and references there for more information. The two fluids of density $\rho_+$ and $\rho_-$ occupy the regions $\Omega_+(t)$ and $\Omega_-(t)$, which are separated by an interface
\[\Sigma(t)=\{(x,f(x,t)): x\in \mathbb{R}\},\]
represented as the graph of a function $f$. This function satisfies (see \cite{CCGS}, equation (1))
\begin{align}
\partial_t f(x,t)& = \frac{\rho_--\rho_+}{2\pi}\mathrm{p.v.}\int \frac{\partial_x f(x,t)-\partial_x f(x',t)}{(x-x')^2+(f(x,t)-f(x',t))^2} (x-x')\,\mathrm{d}x'\label{eq: muskat}\\
f(x,0)&=f_0(x), \ x\in \mathbb{R}.\nonumber
\end{align}
$\Omega_+(t)$ is again assumed to lie above the interface, and we work in the ``stable'' regime $\rho_->\rho_+$. We assume that derivative of the initial data lies in the Wiener algebra $B_0$ (see Definition \ref{def: spaces}), with small norm. Several authors have constructed solutions in the neighborhood of a stationary interface, see \cite{CCGS}, \cite{CG}, \cite{CP}, \cite{SCH}. Although it is a very simple application of the work \cite{DR}, the method used here produces global solutions, analytic for positive time with what seems to be the least regular initial data in the literature.

\section{Equations for vortex sheets}
We begin by presenting the equations for vortex sheets. The interface $\Sigma(t)$ separating the two fluid regions is assumed to be the graph of a function on $\mathbb{R}$:
\[\Sigma(t)=\{(x,y(x,t)):x\in \mathbb{R}\}, \quad t\ge 0.\]
The equations are formulated in terms of the vorticity density $\tilde{\omega}$, defined by:
\begin{displaymath}
\langle \Omega,\phi\rangle = \int_{\mathbb{R}}\tilde{\omega}(x,t) \phi(x,y(x,t)) \,\mathrm{d}x
\end{displaymath}
for all $\phi\in C_c^\infty(\mathbb{R}^2)$. 
 The components of the velocity field can be expressed in terms of $\tilde{\omega}$ by the Biot-Savart law:
\begin{equation}\label{VS:BR}
\left\{ \begin{aligned} v_1(x,t) &= -\frac{1}{2\pi}\textrm{p.v.}\int\frac{y(x,t)-y(x^\prime,t)}{(x-x^\prime)^2+(y(x,t)-y(x^\prime,t))^2}\tilde{\omega}(x^\prime,t)\ud x^\prime \\
   v_2(x,t) &= \frac{1}{2\pi}\textrm{p.v.}\int\frac{x-x^\prime}{(x-x^\prime)^2+(y(x,t)-y(x^\prime,t))^2}\tilde{\omega}(x^\prime,t)\ud x^\prime.
\end{aligned}\right.
\end{equation}

Sulem, Sulem, Bardos and Frisch \cite{SSBF} first derived the equations of motion of a vortex sheet for the Kelvin-Helmholtz problem (see also Sulem, Sulem \cite{SS} for the Rayleigh-Taylor problem):
\begin{equation}\label{VS:Interface}
 \left\{\begin{aligned}
    \partial_{t}y + v_1 y_x &= v_2 \\
    \partial_t(\tfrac{1}{2}\tilde{\omega} - a(v_1+v_2y_x)) + \big(v_1(\tfrac{1}{2}\tilde{\omega} - a(v_1+v_2y_x))\big)_x&\\ \quad - a\big(\big(\frac{\tilde{\omega}^2}{8(1+ y_x^2)}-\frac{|v|^2}{2}-gy\big)\big)_x  &= 0. \\
\end{aligned}\right.
\end{equation}
Throughout, the subscripts $x$ and $t$ denote partial differentiation in the variables $x$ and $t$, respectively.

Following \cite{DR}, we write the system (\ref{VS:BR}), (\ref{VS:Interface}) in terms of its linearization around the stationary solution $(y,\tilde{\omega})=(0,2)$. Introducing the function $\omega$, defined by
\[\tilde{\omega}=2(1+\omega),\]
and differentiating the first equation with respect to $x$, we rewrite (\ref{VS:BR}), (\ref{VS:Interface}) as
\begin{equation}\label{LinNh}
\left\{\begin{aligned}
    \partial_t y_x - \Lambda \omega &= N_1 \\
    \partial_t \omega + aH\partial_t y_x - \Lambda y_x - a\partial_x \omega + agy_x & = N_2 \\
    y_x(t=0) & = y_{x0}.
   \end{aligned}\right.
\end{equation}
$N_1=N_1(y_x,\omega)$ and $N_2=N_2(y_x,\omega)$ are non-linear terms whose exact form will be given in Section 4 (see \eqref{yeqn}, \eqref{weqn} below).
 $\Lambda$ is the pseudo-differential operator with symbol $|\xi|$, and $H$ is the Hilbert transform, with symbol $-i\xi/|\xi|$. For the Fourier transform, we use the definiton
\[\widehat{u}(\xi) =\int e^{-i\xi x} u(x)\,\mathrm{d}x.\]

The equations \eqref{VS:BR} and \eqref{LinNh} form the vortex sheet system with which we shall be concerned.

Before stating our result, let us recall the definition of the function spaces introduced in \cite{DR}. 
\begin{defn} \label{def: spaces} We denote by $B_0$ the space of functions whose Fourier transforms are bounded measures, with the norm
\[\|u\|_{B_0}=\int_\mathbb{R} \mathrm{d}|\widehat u|.\]
For $\rho\ge 0$, the space $B_\rho\subset B_0$ is defined by the norm:
\[\|u\|_{B_\rho}=\int_\mathbb{R} \mathrm{d}|e^{\rho|\xi|}\widehat{u}(\xi)|.\]

For $\alpha \ge 0$, $\mathcal{B}_\alpha$ is the subspace of $C_t^0(\mathbb{R}_{\ge 0};B_0)$ defined by the norm
\[\|u\|_{\mathcal{B}_\alpha}=\inf \{\|\mu\|_{B_0}:\mu \text{ a positive bounded measure}, |e^{\alpha t|\xi|}\widehat{u}(\xi,t)|\le \mu  \text{ for all } t\ge 0\}.\]
The measure achieving the infimum is denoted $|u|_\alpha$.
\end{defn}
The Paley-Wiener theorem ensures that functions in $B_\rho$ extend analytically to the strip $\{|\Im z| <\rho\}$ of the complex plane. $B_\rho$ is an algebra under pointwise multiplication:
\begin{equation}\label{eq: algebra}
\|uv\|_{B_\rho}\le \|u\|_{B_\rho}\|v\|_{B_\rho}.
\end{equation}


The main result of the present work is the following:
\begin{thm}\label{mainthm} Consider $0\leq |a| < 1$. There are $\eps_1 > 0$ and $\alpha>0$ such that for any $y_{0x} \in B_0$ with mean zero and norm no greater than $\eps_1$, there exists a solution $(y_x,\om)$ in $\mathcal{B}_\alpha \times \mathcal{B}_\alpha$ to the system (\ref{LinNh}). The functions $y_x(x,t)$, $\om(x,t)$ converge to zero uniformly as $t$ goes to infinity.

If $a<0$, then for any fixed time $T>0$, there are $\epsilon_2(T)>0$ and $\alpha>0$ such that for any  $y_{0x}\in B_0$ with mean zero and norm no greater than $\epsilon_2$, there exists a pair $(y_x,\om)$ in $\mathcal{B}_\alpha \times \mathcal{B}_\alpha$ which solves the system (\ref{LinNh}) for $0\le t\le T$.
\end{thm} 

As previously mentioned, the case $a=0$ in the theorem was done in \cite{DR}. In contrast to Duchon and Robert for $a=0$, we make no assertion regarding the uniqueness of the solutions provided by Theorem \ref{mainthm}. This is because, although the expressions \eqref{eq: mainsystem} for the solutions $(y_x,\omega)$ are unique up to adding a constant to $\omega$, $\omega$ is no longer characterized by having zero mean, as it was in \cite{DR}.

We begin the proof of the theorem in the case $a>0$ in the next section, where we further reduce the system (\ref{LinNh}) to a form suitable to the application of a contraction mapping argument in $\mathcal{B}_\alpha$. In Sections \ref{sec: structure} and \ref{sec: nonlinear}, we explain how to adapt the estimates in \cite{DR} to the non-linear terms appearing in the equations for $a\neq 0$. In Section \ref{sec: agpos} we show how to address the case $a<0$, and establish the second part of the theorem. Finally, in Section \ref{sec: muskat}, we write the Muskat equation in a form where the estimates of Duchon and Robert can be applied.

\section{The linear system}

We initially treat \eqref{LinNh} as an inhomogeneous linear system. That is, $N_1$ and $N_2$ are assumed to be given functions in $\mathcal{B}_\alpha$, $\alpha > 0$, and the system is solved by diagonalization. First, consider the case $N_1= N_2 =0$. We let $V=(y_x,\omega)$. Substituting the first equation in \eqref{LinNh} into the second, we find the system:
\begin{equation}\label{eq: relation}
\partial_t V(t) = AV(t),
\end{equation}
where the operator matrix $A$ is given by
\begin{equation}
\left(\begin{array}{cc}
0 &\Lambda\\
\Lambda-ag &2a\partial_x
\end{array}\right).
\end{equation}
Taking Fourier transforms, we obtain the matrix of symbols 
\begin{equation}
\widehat{A}(\xi)= \left(\begin{array}{cc}
0 &|\xi|\\
|\xi|-ag &2 ai\xi
\end{array}\right).
\end{equation}
The eigenvalues are
\begin{equation}\label{eq: evalues}
\lambda_\pm(\xi) = ai\xi\pm |\xi| \sqrt{1-a^2-ag|\xi|^{-1}}.
\end{equation}
This leads us to introduce the notation
\[\Del =  \sqrt{1-a^2-ag\Lambda^{-1}}.\]
Let $m(\xi)$ be the multiplier corresponding to $\Del$:
\[\widehat{\Del f}(\xi)=m(\xi)\widehat{f}(\xi).\]
When $1-a^2-ag|\xi|^{-1}$ is negative, as can occur if $a<0$, the square root is defined as
\[m(\xi)= i\sqrt{|1-a^2-ag|\xi|^{-1}|}.\]
In case the $a<0$, $m(\xi)=0$ when
$$ |\xi|= \frac{ag}{1-a^2}. $$
In the case $a>0$ which we consider in this section, the multiplier $m(\xi)$ is bounded away from zero. $\Del^{-1}$ is then well-defined and bounded on $B_0$. 

The eigenvectors corresponding to \eqref{eq: evalues} are
\begin{equation}
r_\pm(\xi) =\left(1,\frac{1}{|\xi|}(ia\xi \pm |\xi|m(\xi))\right).
\end{equation}
We can now diagonalize \eqref{eq: relation}. Let
\begin{equation}\label{eq: BV}
U=\left(\begin{array}{c} u_+\\ u_- \end{array}\right) = BV,
\end{equation}
\begin{equation}
B = \left(\begin{array}{cc}
-\Del - aH & -1\\
-\Del+ aH &  1
\end{array}\right).
\end{equation}
The variables $U=(u_+,u_-)^t$ solve the system
\begin{equation} \label{eq: diagsystem}
\partial_t \widehat{U}(\xi,t) = \left(\begin{array}{cc} \lambda_+(\xi) & 0\\
 0 &\lambda_-(\xi)\end{array}\right)\widehat{U}(\xi,t).
\end{equation}
Writing
\begin{equation}\label{eq: Ssemigroup}
S_\pm(t) = e^{t(ia\partial_x\pm \Lambda \mathfrak{M})},
\end{equation}
the solution to \eqref{eq: diagsystem} is
\[U(t) = (S_+(t)u_+(0), S_-(t)u_-(0)).\]

We now deal with the inhomogeneous case where $N_1$ and $N_2$ in \eqref{LinNh} are functions in $\mathcal{B}_\alpha$. On the Fourier side, the system may be written as
\begin{equation}\label{eq: VnonH}
\partial_t \widehat{V}(t) = \left( \begin{array}{cc}
0 & |\xi|\\
|\xi|-ag & 2ia\xi
\end{array}\right)\widehat{V}(t) + \left(\begin{array}{c} \widehat{N}_1(t) \\ \widehat{N}_2(t) \end{array}\right).
\end{equation}
For the variables $U$ \eqref{eq: BV}, we find
\begin{equation}\label{eq: UnonH}
\partial_t \widehat{U}(t) = \left( \begin{array}{cc}
0 & |\xi|\\
|\xi|-ag & 2ia\xi
\end{array}\right)\widehat{U}(t) + \left(\begin{array}{c} -\big(m(\xi)-ai\frac{\xi}{|\xi|}\big)\widehat{N}_1(t) -\widehat{N}_2(t)\\ -\big(m(\xi)+ai\frac{\xi}{|\xi|}\big)\widehat{N}_1(t)+\widehat{N}_2(t) \end{array}\right).
\end{equation}
Following \cite{DR}, the solutions of \eqref{eq: UnonH} are expressed through the Duhamel formula. To avoid the growth due to the eigenvalue $\lambda_+$, Duchon and Robert use an integral extending from $t$ to infinity (see \eqref{eq: Iplus}, \eqref{eq: usols}), effectively prescribing the behaviour at temporal infinity of the solution. This prescription is already implicit in the choice of the spaces $\mathcal{B}_\alpha$. We remark that the consistency of the assumption that the solutions decay for large time is a special feature of the problem. It is a manifestation of the ellipticity of the Rayleigh-Taylor problem in space and time (see \cite{BL}).

To write down the solutions to \eqref{eq: UnonH}, we define
\begin{equation}\label{eq: Iplus}
I^+h(x,t) = \int_t^\infty S_+(t-s)h(x,s)\,\mathrm{d}s,
\end{equation}
and
\begin{displaymath}
I^-h(x,t) = \int_0^tS_-(t-s)h(x,s)\,\mathrm{d}s,
\end{displaymath}
and finally:
\begin{displaymath}
I_0h = I^+h(0).
\end{displaymath}
Note that the operators $I^{\pm}$ defined here differ from the operators $I^{\pm}$ defined in \cite{DR}.

For $N_1$ and $N_2$ in $\mathcal{B}_\alpha$, the solutions to \eqref{eq: UnonH} are given by
\begin{align}\label{eq: usols}
u_+(t) &= S_+(t) u_+(0) - I^+(N_2+(\Del+aH)N_1)(t) + S_+(t)I_0(N_2+(\Del+aH) N_1),\\
u_-(t) &= S_-(t)u_-(0) + I^-(N_2-(\Del-aH) N_1)(t).\label{eq: usols2}
\end{align}
At this point, initial data has only been prescribed for $y_x$ in \eqref{LinNh}, and not for $\omega$. The remaining degree of freedom is used to ensure that $u_+$ belongs to $\mathcal{B}_0$. Namely, we set
\[u_+(0) = -I_0(N_2+(\Del+aH) N_1).\] 
This choice amounts to the prescription
\begin{equation} \label{eqn:wandy}
\omega(0) = -aHy_{x0}-\Del y_{x0}+I_0(N_2+(\Del+aH)N_1).
\end{equation}
Taking equation (\ref{eqn:wandy}) into account, we express $y_x$ and $\omega$ in terms of $u_+$ and $u_-$. This results in the following representation of the solutions:
\begin{equation} \label{eq: mainsystem}
\left(\begin{array}{c}
y\\
\omega
\end{array}\right)= B^{-1}U(t) = -\frac{1}{2}\Del^{-1}\left(\begin{array}{cc}1 & 1\\ \Del-aH & -\Del-aH \end{array}\right)U(t),
\end{equation}
with $U=(u_+,u_-)$ given by \eqref{eq: usols}, \eqref{eq: usols2}, with 
\begin{align*}
u_+(0) &= -I_0(N_2+(\Del+aH) N_1)\\
u_-(0) &= (-\Del+aH)y_{x0}+\omega(0) = -2\Del y_{x0}+I_0(N_2+(\Del+aH)N_1).
\end{align*}
As previously remarked, when $a>0$, $\Del^{-1}$ is an operator with bounded multiplier. We will apply a fixed point argument to the system \eqref{eq: mainsystem}.

Before discussing the nonlinear terms, let us derive two simple estimates for the operators $I^+$ and $I^-$. We will see in Sections \ref{sec: structure} and \ref{sec: nonlinear} that $N_1$ and $N_2$ have the form
\begin{align} \label{nonlinearterms}
N_1& = F(y_x,\omega)_x,\\
N_2 &= G_1(y_x,\omega)_t+G_2(y_x,\omega)_x,
\end{align}
where $F(y_x,\om)$, $G_1(y_x,\om)$ and $G_2(y_x,\om)$ are in $\mathcal{B}_\alpha$ if $y_x$ and $\om$ are.
This motivates the next lemma.

\begin{lemma}\label{lem: linearestimates} Suppose $a>0$. There exists $\alpha_0>0$ such that for $0<\alpha <\alpha_0$ and any $F\in \mathcal{B}_\alpha$, we have the estimates:
\begin{align}
 |I^\pm (F_x)|_\alpha &\leq C(\alpha)|F|_\alpha, \label{eq: xlinest}\\ 
 |I^\pm (F_t)|_\alpha &\leq C(\alpha)|F|_\alpha. \nonumber
\end{align}
\end{lemma}
\begin{proof}
We will prove the estimates for the operator $I^-$, the estimates for $I^+$ follow similarly.
\begin{align*}  
 & e^{\alpha t|\xi|}|\widehat{I^-(F_x)}(\xi,t)|\\  
\le&\   \int_0^t \left|e^{(t-s)(ia\xi+\alpha|\xi|-m(\xi)|\xi|)}\right||\xi| e^{\alpha s|\xi|}|\widehat{F}(\xi,s)| \ud s \\  
 \leq&\ |F|_\alpha\int_0^t e^{(t-s)|\xi|(\alpha-m(\xi))}|\xi| \ud s \\  
=&\   |F|_\alpha\frac{1}{|\alpha-m(\xi)|}(1-e^{t|\xi|(\alpha-m_{a,g})}).
\end{align*}
When $ag < 0$, the last line is uniformly bounded in $t$ if $\alpha< \sqrt{1-a^2}/2$.

To derive the analogous bound for $I^-(F_t)$, we integrate by parts in $s$:
\begin{align*}
 &e^{\alpha t|\xi|}|\widehat{I^-(F_t)}(\xi,t)|\\
=&\ \left|e^{\alpha|\xi|t}\int_0^te^{(t-s)(ia\xi-m(\xi)|\xi|)}\partial_s \widehat{F}(\xi,s)\ud s\right|\\
\le&\  \Big| e^{\alpha|\xi|t}\cdot e^{(t-s)(ia\xi-m(\xi)|\xi|)} \widehat{F}(\xi,s)\Big|^{s=t}_{s=0}\\
&\quad + e^{\alpha |\xi|t}\Big|\int_0^t(ia\xi-m(\xi)|\xi|)e^{(t-s)(ia\xi-m(\xi)|\xi|)}\widehat{F}(\xi,s)\ud s\Big|\\
\le&\  e^{\alpha|\xi|t}|\widehat{F}(\xi,t)|+e^{(\alpha-m(\xi))|\xi|t}|\widehat{F}(\xi,0)|\\
&\quad+ |F|_\alpha \int_0^t|\xi|(|a|+m(\xi))e^{(t-s)(\alpha-m(\xi))|\xi|} \ud s\\
\le&\  C(\alpha)|F|_\alpha,
\end{align*}
where we have again assumed $\alpha < \sqrt{1-a^2}/2$ to pass to the final line. 

To reproduce the estimate above for $I^+(F_t)$, we must integrate by parts in time. The boundary term at $t=\infty$ is zero by the assumption $F\in \mathcal{B}_\alpha$, $\alpha>0$.

\end{proof}
The estimates in the previous lemma also hold if $I^\pm$ is replaced by $S_-(t)I_0$.

\section{The Structure of the Non-Linear Terms}\label{sec: structure}
We consider the structure of the non-linear terms $N_1(y_x,\om)$ and $N_2(y_x,\om)$ appearing in the system (\ref{LinNh}). From the equations (\ref{VS:Interface}) and (\ref{LinNh}), we find:
\begin{equation}\label{yeqn}
N_1(y_x,\om) = (-v_1y_x + (v_2 - H\om))_x
\end{equation}
and
\begin{equation}\label{weqn}
N_2(y_x,\om) = G_1(y_x,\om)_t + G_2(y_x,\om)_x,
\end{equation}
where
\begin{align}
G_1(y_x,\om)&=(av_1+aHy_x)+av_2y_x, \label{eq: G1def}\\
G_2(y_x,\om)&=\om v_1 + v_1 + Hy_x - av_1^2 - av_1v_2 y_x \label{eq: G2def}
\\& - a\left(\left(\frac{1}{2}\frac{(1+\om)^2}{1+y_x^2} - \om\right) - \frac{v_1^2}{2} - \frac{v_2^2}{2}\right).\nonumber
\end{align}
We have grouped together terms with their linear part.

We note that $N_1(y_x,\om) = F(y_x,\omega)_x$ has no dependence on the parameter $a$. $N_2(y_x,\om)$ depends on $a$ and is different from the function $G(y_x,\om)_x$ appearing in \cite{DR}, p. 217.

For simplicity, we will suppress the dependence on $t$ in the remainder of this section and in the next section. For example, we will write $y(x)$ to mean $y(x,t)$.

Define
\[ p = p(x,x') = \frac{y(x) - y(x')}{x-x'}.\]
With this notation, the equations (\ref{VS:BR}) become
\begin{equation}\label{v1v2}
\left\{\begin{aligned} 
v_1(x) &= -\frac{1}{\pi} \text{ p.v.} \int\frac{p}{1+p^2}\frac{1+\om(x')}{x-x'} \ud x',\\
v_2(x) &= \frac{1}{\pi} \text{ p.v.} \int\frac{1}{1+p^2}\frac{1+\om(x')}{x-x'} \ud x'.
\end{aligned}\right. 
\end{equation}
Recall the following singular integral representations for the operators $H$ and $\Lambda$:
\[ \La u(x) = \frac{1}{\pi} \text{p.v.} \int \frac{u(x)-u(x')}{(x-x')^2} \ud x' \qquad \text{and} \qquad Hu(x) = \frac{1}{\pi}\text{p.v.} \int \frac{u(x')}{x-x'} \ud x'.\]

Duchon and Robert expand $F(y_x,\om)$, $G_1(y_x,\om)$ and $G_2(y_x,\om)$ in terms of the sequence of singular integral operators $\{T_j\}_{j\ge 1}$ defined by
$$T_j(y_x)u(x) :=\frac{1}{\pi} \text{ p.v.} \int \bigg(\frac{y(x) - y(x')}{x-x'}\bigg)^j \frac{u(x')}{x-x'} \ud x'.$$
An important point below is that the linear terms resulting from the expansion of $v_1$ and $v_2$ in sums of $T_j(y_x)$ cancel the non-constant linear terms appearing in \eqref{eq: G1def} and \eqref{eq: G2def}: the non-linear terms in \eqref{LinNh} are essentially quadratic for small $y_x$ and $\omega$ in $B_0$.

For $F(y_x,\om)$, we have, just as in \cite{DR}, p. 217:
\[F(y_x,\om) = \frac{1}{\pi} \text{ p.v.}\int \bigg(\frac{1}{1+p^2} - 1\bigg)\frac{1+\om(x')}{x-x'} \ud x' + \frac{1}{\pi} y_x \text{ p.v.}\int \frac{p}{1+p^2}\frac{1+\om(x')}{x-x'} \ud x',\]
and thus using, formally at first, the expansion
\[\frac{1}{1+p^2}=\sum_{n= 0}^\infty (-1)^np^{2n},\]
we find the expression
$$ F(y_x,\om) = \sum_{j=2}^\infty \eps_jT_j(y_x)(1+\om) + y_x \sum_{j=1}^\infty \eps'_j T_j(y_x)(1+\om),$$
where $\eps_j$, $\eps'_j$ equals $-1$, $0$ or $1$, see \cite{DR}, p. 219. The estimate \eqref{Tjbound1} below implies that the series converges in $B_0$ norm for $\|y_x\|_{B_0}<1$ and $\omega\in B_0$. Note that the terms of lowest order in this expansion are quadratic in $y_x$ and $\omega$.

We now show that $G_1(y_x,\om)$ and $G_2(y_x,\om)$ have a similar decomposition. Using the expression for $v_1(x)$ in (\ref{v1v2}), the term in parentheses in \eqref{eq: G1def} is
\begin{equation} \label{G1}
av_1 + aHy_x = -aT_1(y_x)\om + a\sum_{j=3}^\infty \eps_j T_j(y_x)(1+\om).
\end{equation}
In particular, the lowest order term in the expansion for $av_1$ is $-aHy_x$. Using the expression for $v_2(x)$, the other term in $G_1(y_x,\om)$ satisfies
$$av_2y_x = y_x H\om + y_x\sum_{j=2}^{\infty} \eps_j T_j(y_x)(1+\om).$$
The series in the previous equations will be seen to converge in norm as a consequence of the non-linear estimate \eqref{Tjbound1}, provided $\|y_x\|_{B_0}<1$. As was the case for $F(y_x,\omega)$, every term in the expression for $G_1(y_x,\om)$ either contains a factor of $T_j(y_x)(1+\om)$ for some $j\geq 2$, or else at least two factors of $y_x$ and $\om$. That is, $G_1$ is quadratic in $y_x$ and $\omega$ around $(0,0)$.

Turning to $G_2(y_x,\om)$, we expand the terms $\omega v_1$, $v_1+Hy_x$, $av_1^2$ and $-\frac{1}{2}v_1^2-\frac{1}{2}v_2^2$ using \eqref{v1v2}. In the term
\begin{equation}\label{eq: omexp}
\frac{1}{2}\frac{(1+\om)^2}{1+y^2_x}-\omega=\frac{1}{2}\frac{1+2\om+\om^2}{1+y_x^2}-\omega,\end{equation}
we use the expansion
\[\frac{1}{1+y_x^2}=\sum_{n=0}^\infty(-1)^ny_x^{2n}.\]
This converges in $B_0$ by the algebra property \eqref{eq: algebra}, provided $\|y_x\|_{B_0}<1$. After this expansion, the non-constant linear term $\om$ in \eqref{eq: omexp} disappears. The remaining term in \eqref{eq: G2def} is
 \begin{equation} \label{G2} av_1v_2y_x = ay_x \bigg(\sum_{j=1}^{\infty} \eps_j T_j(y_x)(1+\om)\bigg)\bigg(H\om + \sum_{j=2}^\infty \eps_j'T_j(y_x)(1+w) \bigg).
 \end{equation}
Each term in $G_2(y_x,\om)$ either contains a factor of $T_j(y_x)(1+\om)$ for some $j\geq 2$ or else contains a factor of $T_1(y_x)$ together with at least one other factor of $y_x$, $\om$ or $T_1(y_x)$.

 \section{The non-linear estimate}\label{sec: nonlinear}
The following is the key estimate in Duchon and Robert's work:
\begin{lemma}[\cite{DR}, Section 4] \label{mainlem} Let $y_{1x}$, $y_{2x}$, $\om_1$, $\om_2$ be elements of $B_{\rho}$. Assume $|e^{r\rho}\widehat{y_{ix}}| \leq \mu$, $|e^{r\rho}\widehat{\om_i}| \leq \mu$ ($i=1,2$) and $|e^{r\rho}(\widehat{y_{1x}} -\widehat{ y_{2x}})| \leq \nu$, $|e^{r\rho}(\widehat{\om_1} - \widehat{\om_2})| \leq \nu$ with $\mu$ and $\nu$ positive bounded measures, $\int \ud \mu < 1$. Then,
$$|e^{r\rho}(\widehat{F}(y_{1x}, \om_1) - \widehat{F}(y_{2x},\om_2))| \leq A(\mu)*\nu,$$
where $A$ is a continuous map from the open unit ball of $\mathcal{M}_+$ (the set of bounded positive measures) into $\mathcal{M}_+$, with $A(0) = 0$, and $e^{r \rho}$ denotes the function $\xi \mapsto e^{\rho |\xi|}$ on $\mathbb{R}$. All inequalities above hold in the sense of measures.
\end{lemma}

We show in this section that an analogous inequality for $G_1(y_x,\om)$ and $G_2(y_x,\om)$ is true. To prove the above lemma, Duchon and Robert expressed $F(y_{1x},\om_1) - F(y_{2x},\om_2)$ in terms of the operators $R_k(y_{1x},\ldots, y_{kx})$, defined by
$$R_k(y_{1x},\ldots,y_{kx}) \Om(x) := \frac{1}{\pi} \text{ p.v.} \int \frac{1}{k}(p_1\cdots p_k)_x \Om(x')\ud x',$$
where $p_i(x,x') = (y_i(x)-y_i(x'))/(x-x')$. Abusing notation, we will denote
$$R_k(y_x) := R_k(y_x,\ldots,y_x).$$
We need the following result from \cite{DR}, p. 220. In the proposition below and in the rest of the paper the symbol $\mathcal{F}$ denotes the Fourier transform: $\left(\mathcal{F}(f)\right)(\xi)=\widehat{f}(\xi)$.
\begin{prop}\label{eq: drprop} Let $y_{1x},\ldots,y_{kx}$ be elements of $B_\rho$. Then $R_k(y_{1x},\ldots, y_{kx})$ maps $B_0$ into $B_\rho$, and
\[|e^{r \rho}\mathcal{F}(R_k(y_{1x},\ldots,y_{kx})\Omega)| \leq 2|e^{\rho r}\widehat{y_{1x}}| * \cdots * |e^{\rho r}\widehat{y_{kx}}| * |e^{-\rho r}\widehat{\Omega}|.\]
\end{prop}

From the previous proposition, we deduce the appraisals (see \cite{DR}, pp. 222-223)
\begin{equation} \label{Tjbound1} |e^{r\rho}\mathcal{F}(T_j(y_x)\Om)| \leq (1+2j)|e^{r\rho}\widehat{y_x}|^{*j}*|e^{r\rho}\widehat{\Om}|,
\end{equation}
and
\begin{equation} \label{Tjbound2} |e^{r\rho}\mathcal{F}(T_j(y_{1x})\Om_1 - T_j(y_{2x})\Om_2)| \leq c(j)(\mu^{*(j-1)} + \mu^{*j})*\nu,
\end{equation}
where $c(j) := 2j^2 +3j+1$ and $\mu^{*j} := \mu*\cdots*\mu$ ($j$ times).

For the terms
\begin{align*}
av_2y_x&= ay_x\big(H\omega+\sum_{j=2}^\infty \eps_jT_j(y_x)(1+\om)\big), \\
a(v_1+Hy_x)&=-aT_1(y_x)\om + a\sum_{j=3}^\infty \eps_j T_j(y_x)(1+\om),\\
\end{align*}
in $G_1(y_x,\om)$ and 
\begin{align*}
\omega v_1&= -\omega\sum_{j=1}^\infty \eps'_j T_j(y_x)(1+\om),   \\ 
v_1+Hy_x, &= -T_1(y_x)\om + \sum_{j=3}^\infty \eps_j T_j(y_x)(1+\om),\\  
-a\big(\frac{1}{2}\frac{(1+\omega)^2}{1+y_x^2}-\omega\big)&=-\frac{a}{2}-a\frac{\omega^2}{2}\\
&\quad -\frac{a}{2}(1+2\omega+\omega)^2\sum_{n= 1}^\infty(-1)^ny_x^{2n},
\end{align*}
in $G_2(y_x,\om)$, whose expansions do not involve products of two sums over the operators $T_j(y_x)$, we proceed in an identical fashion to Duchon and Robert in their proof of the estimate for $F(y_x,\omega)$ on p. 222 in \cite{DR} to establish Lemma \ref{mainlem}.

We are left with terms which do involve the product of two sums over $T_j(y_x)$:
\begin{align*}
-\frac{1}{2}v_1^2-\frac{1}{2}v_2^2&= -\frac{1}{2}\omega\big(\sum_{j=1}^\infty \eps'_j T_j(y_x)(1+\om)\big)^2\\
&\quad -\frac{1}{2}\big(\sum_{j=0}^\infty \eps_jT_j(y_x)(1+\om)\big)^2,\\
av_1^2&=a\omega\big(\sum_{j=1}^\infty \eps'_j T_j(y_x)(1+\om)\big)^2,
\end{align*}
and the term $av_1v_2y_x$ appearing in equation \eqref{G2}. We deal with the latter term. All others terms are estimated similarly. Firstly, consider
$$ \Psi_1(y_x,\om) := ay_x \cdot  H\om \cdot \bigg(\sum_{j=1}^\infty \eps_j T_j(y_x)(1+\om)\bigg). $$
Since $H$ is an isometry in $B_\rho$, by equations (\ref{Tjbound1}) and (\ref{Tjbound2}), we have
$$ |e^{r\rho} \mathcal{F}\big(\Psi_1(y_{1x},\om_1) - \Psi_1(y_{2x},\om_2)\big)| \leq A_1(\mu)*\nu,$$
for a continuous function $A_1$ satisfying $A_1(0) = 0$.

We also need to consider
$$ \Psi_2(y_x,\om) := ay_x\cdot\bigg(\sum^\infty_{j=1}\eps_jT_j(y_x)(1+\om)\bigg) \cdot \bigg(\sum_{j=2}^\infty \eps'_jT_j(y_x)(1+\om)\bigg).$$
Using equations (\ref{Tjbound1}) and (\ref{Tjbound2}) again, the same inequality holds for $\Psi_2$, with  some continuous function $A_2$ satisfying $A_2(0) = 0$ on the right side.

To summarize, for $|e^{r\rho}\widehat{y_{ix}}|$, $|e^{r\rho}\widehat{\om_i}| \leq \mu$ and $|e^{r\rho}(\widehat{y_{1x}}-\widehat{y_{2x}})|$, $|e^{r\rho}(\widehat{\om_1} - \widehat{\om_2})| \leq \nu$, we have
\begin{align}
\label{Fineq} |e^{r\rho}(\widehat{F}(y_{1x},\om_1) - \widehat{F}(y_{2x},\om_2))| &\leq A(\mu) * \nu \\
\label{Gineq} |e^{r\rho}\mathcal{F}\big(G_1(y_{1x},\om_1)+G_2(y_{1x},\om_1) - G_1(y_{2x},\om_2) - G_2(y_{2x},\om_2)\big)| &\leq A(\mu)*\nu,
\end{align}
for a continuous function $A$ satisfying $A(0) = 0$.

The estimates (\ref{Fineq}) and (\ref{Gineq}) allow for the construction of solutions to (\ref{eq: mainsystem}) with prescribed initial data $y_{0x}\in B_0$ and $\omega(0)$ given by (\ref{eqn:wandy}) by a contraction argument in $\mathcal{B}_\alpha \times \mathcal{B}_\alpha$ around the linear solution
\begin{align*}
y_{x,\mathrm{lin}} &= S_-(t)y_{x_0}\\
\omega_{\mathrm{lin}} &= -aHS_-(t)y_{x_0}-S_-(t)(\Del y_{x0}).
\end{align*}
This concludes the proof of the first part of Theorem \ref{mainthm}.
\section{The case $a<0$}\label{sec: agpos}
In this section, we prove the second part of Theorem \ref{mainthm}. When $a<0$, we cannot reproduce the estimates in Lemma \ref{lem: linearestimates} because the linear semigroup does not dampen small frequencies in time. Nevertheless, the quadratic nature of the nonlinearity about $(y_x,\omega)=(0,0)$ allows us to obtain a local in time result.

Another difference with the case $a>0$ is that the multiplier $m(\xi)$ is zero when $|\xi| = ag/(1-a^2)$. At this frequency, the linear homogeneous system with $N_1=N_2=0$ has solution matrix
\[A(t)=e^{ia\xi t}\left(\begin{array}{cc}
1-ia\xi & |\xi|t\\
a^2|\xi| &  1+ia\xi 
\end{array}\right).\]
$A(t)$ has a double eigenvalue $e^{ia\xi t}$ and its Jordan canonical form consists of a $2\times 2$ block. This prevents us from defining the inverse of $\Del$, and diagonalizing the system as we did previously. 

Since there are only two problematic frequencies, if one is content with solutions existing locally in time, it is sufficient to note that when 
\begin{equation}\label{eq: bdfreq} 
|\xi| \le 2ag/(1-a^2)
\end{equation}
we have
\begin{equation}\label{eq: trivial}
\|e^{t\widehat{A}(\xi)}\|\le e^{Ct}.
\end{equation}
Solving \eqref{eq: VnonH}, we have
\begin{equation}\label{eq: direct}
\widehat{V}(\xi,t) = e^{t\widehat{A}(\xi)}\widehat{V}(0) + \int_0^t e^{(t-s)\widehat{A}(\xi)}\left(\begin{array}{c} \widehat{N}_1(s)\\ \widehat{N}_2(s)\end{array}\right)\,\mathrm{d}s.
\end{equation}
We separate the nonlinear term as
\begin{equation}
\left(\begin{array}{c}\widehat{N}_1\\ \widehat{N}_2 \end{array}\right) = \left(\begin{array}{c} i\xi \widehat{F}  \\ i\xi \widehat{G_2} \end{array}\right) +\left(\begin{array}{c}0 \\ (\widehat{G_1})_s \end{array}\right).
\end{equation}
Recalling the assumption \eqref{eq: bdfreq}, the contribution from the first component is bounded in $\mathcal{B}_\alpha$ norm by
\[\int_0^t Ce^{Ct}(|F(\xi,s)|+|G_2(\xi,s)|)\,\mathrm{d}s\le Ce^{C'(\alpha)t}\big(|F|_\alpha+|G_2|_\alpha\big).\]
For the second component, we first integrate by parts in $s$ to find
\[\left(\begin{array}{c} 0\\ \widehat{G}_1(t)\end{array}\right) - e^{t\widehat{A}(\xi)}\left(\begin{array}{c}0\\ \widehat{G}_1(0)\end{array}\right) +\int_0^t \widehat{A}(\xi)e^{(t-s)\widehat{A}(\xi)}\left(\begin{array}{c}0\\ \widehat{G}_1(s)\end{array}\right)\,\mathrm{d}s.\]
This shows that the $\mathcal{B}_\alpha$ norm of the second component can be bounded by 
\[Ce^{C''(\alpha)t}|G_1|_\alpha.\]

Now fix $T>0$, and let $\chi\in C^\infty(\mathbb{R})$ be such that $0\le \chi \le 1$, $\chi(s)=1$, for $|s|\le 1$, and  $\chi(s)=0$ for $|s|\ge 2$. Define
\begin{align}
\tilde{I}^+ h(x,t)&=\int_t^\infty \chi(s) S_+(t-s)h(x,s)\,\mathrm{d}s\\
\tilde{I}^-h(x,t)&= \int_0^t \chi(s) S_-(t-s)h(x,s)\,\mathrm{d}s\\
\tilde{I}_0h&= I^+h(x,0).
\end{align}
When $t\ge T$ and $|\xi|> 2ag/(1-a^2)$, the Fourier transforms of \eqref{eq: usols}, \eqref{eq: usols2}, with $\tilde{I}^\pm$ replacing $I^\pm$ still provide a solution to \eqref{eq: UnonH}. Recall the notation $\mathcal{F}h = \widehat{h}$.
Defining $\tilde{U}=(\tilde{u}_+,\tilde{u}_-)$ by
\begin{align}
\label{eq: usols3}
\mathcal{F}\tilde{u}_+(t) &= \mathcal{F}S_+(t)\widehat{u}_+(0) - \mathcal{F}\tilde{I}^+(N_2+(\Del+aH)N_1)(t)\\
&\quad + \mathcal{F}S_+(t)\tilde{I}_0(N_2+(\Del+aH) N_1), \nonumber \\
\mathcal{F}\tilde{u}_-(t) &= \mathcal{F}S_-(t)\widehat{u}_-(0) + \mathcal{F}\tilde{I}^-(N_2-(\Del-aH) N_1)(t)\label{eq: usols4},
\end{align}
 we obtain the representation, valid for $t\le T$
\begin{equation}\label{eq: solT}
\widehat{V}(\xi,t)=\left(\begin{array}{c} \widehat{y}_x \\ \widehat{\omega} \end{array}\right)= \widehat{B}^{-1}\mathcal{F}\tilde{U}(\xi,t).
\end{equation}
The restriction on the frequency ensures that the multiplier $m(\xi)^{-1}$ appearing in $\widehat{B}^{-1}$ is bounded above. One can now reproduce the estimates for the case $a>0$ exactly. Combining this with the treatment of the frequencies $|\xi|\le 2ag/(1-a^2)$, we have established the following:
\begin{prop} There exist $C,\alpha>0$ such that if $\|y_{0x}\|_{B_0}<\epsilon$ and $T=T(\|y_{0x}\|_{B_0})$ is sufficiently small, there is a pair $(y_x(t),\omega(t))\in \mathcal{B}_\alpha$ such that $\widehat{V}(t)=(\widehat{y}_x(t),\widehat{\omega}(t))$ solves \eqref{eq: direct} when $|\xi|\le 2ag/(1-a^2)$, $t\le T$ and solves \eqref{eq: solT} for $|\xi|\ge 2ag/(1-a^2)$ for $t\le T$ and $y_x(0)=y_{x0}$. 
\end{prop}
The second part of Theorem \ref{mainthm} follows at once. We remark that the treatment of low frequencies we have given here results in a time of existence $T$ of order $\log \frac{1}{\epsilon}$ if the initial data has size $\epsilon$. It is possible to improve on this somewhat and obtain an estimate of the type $T\sim \frac{1}{\sqrt{\epsilon}}$ by treating the frequencies $|\xi|= \frac{ag}{1-a^2}$ separately instead of using the crude estimate \eqref{eq: trivial}. We refrain from doing this because it results in a significantly longer proof, and the added benefit is not clear.

\section{Small global solutions for the Muskat problem}\label{sec: muskat}
In this section, we show that the non-linear estimate of Duchon and Robert in Proposition \ref{eq: drprop} can also be used to construct global solutions to the Muskat equation \eqref{eq: muskat} in the neighborhood of a flat interface.

To construct global solutions with $(f_0)_x\in B_0$, we proceed analogously to the Rayleigh-Taylor case, although the computation is simpler. We assume we are in the ``stable'' configuration with the fluid of higher mass density lying below the fluid of lower density:
\[\rho_->\rho_+.\]
We differentiate \eqref{eq: muskat} to obtain an equation for $f_x(x,t)$, which we write in perturbative form around the flat interface $f\equiv 0$:
\begin{align}
\partial_t f_x(x,t) + \frac{(\rho_--\rho_+)}{2}\Lambda f_x(x,t) = \partial_x N(f)(x,t) \label{eq: muskat2}\\
f_x(x,0) =(f_0)_x(x),
\end{align}
where the nonlinearity is given by
\[N(f)(x,t) =  -\frac{\rho_--\rho_+}{2\pi}\int \frac{\partial_xf(x,t)-\partial_x f(x',t)}{x-x'}\cdot \frac{\left(\frac{f(x,t)-f(x',t)}{x-x'}\right)^2}{1+\left(\frac{f(x,t)-f(x',t)}{x-x'}\right)^2}\,\ud x.\] 
We expand $N(f)$ in terms of operators similar to the $T_j$ defined in Section \ref{sec: nonlinear}:
\[N(f)(x,t) = \sum_{j\ge 2}\epsilon_j \left(\tilde{T}_j(f_x)f_x\right)(x,t),\]
where
\[\left(\tilde{T}_j(f_x)u\right)(x,t)= \frac{1}{\pi} \int \left(\frac{f(x,t)-f(x',t)}{x-x'}\right)^j\frac{u(x,t)-u(x',t)}{x-x'}\ud x'\]
and $\epsilon_j \in \{-1,0,1\}$. The estimate \eqref{Tjbound1} (with $\tilde{T}_j$ replacing $T_j$ and $f_x$ replacing $\Omega$) shows that the expansion is convergent in $B_0$ if $\|f_x\|_{B_0}<1$.

Returning to equation \eqref{eq: muskat2}, we can express the solution in Duhamel form
\[f(x,t) = e^{-t\big(\frac{\rho_--\rho_+}{2}\big)\Lambda}f_0(x) +\int_0^te^{-(t-s)\big(\frac{\rho_--\rho_+}{2}\big)\Lambda}\left(N(f)(x,s)\right)_x\ud s.\]
Since $\rho_--\rho_+>0$, we can now reproduce the analogue of the linear estimate \eqref{eq: xlinest} in Lemma \ref{lem: linearestimates}, and combine it with Duchon and Robert's nonlinear estimate \eqref{Tjbound1}. Applying a contraction argument in $\mathcal{B}_\alpha$ for $\alpha < (\rho_--\rho_+)/2$, we obtain the following
\begin{thm}
There are constants $\epsilon>0$ and $\alpha>0$ such that, for any initial data with $\|(f_0)_x\|_{B_0}\le \epsilon$, there is a solution to the Muskat problem \eqref{eq: muskat2} in $\mathcal{B}_\alpha$. The solution is unique in a ball in $\mathcal{B}_\alpha$.
\end{thm}

\textbf{Acknowledgement.} We wish to thank the anonymous referee for their patience and valuable criticism.

\end{document}